\documentclass[11pt,reqno]{amsart}
\usepackage{fullpage} % Package to use full page
\usepackage{parskip} % Package to tweak paragraph skipping
\usepackage{tikz} % Package for drawing
\usepackage{amsmath,amssymb,amsfonts,amsthm}
\usepackage{hyperref}
\usepackage{longtable}
\usepackage{indentfirst}
\theoremstyle{definition}
\usepackage{array} % for extrarowheight
\newcommand{\Q}{\mathbf{Q}}
\newcommand{\Z}{\mathbb{Z}}

\newtheorem{theorem}{Theorem}\newtheorem{conj}{Conjecture}\newtheorem{lem}{Lemma}
\newtheorem{example}{Example}\newtheorem{rmk}{Remark}
\DeclareMathOperator{\Rep}{Rep}

\DeclareMathOperator{\SU}{SU}
\DeclareMathOperator{\U}{U}
\DeclareMathOperator{\SO}{SO}
\DeclareMathOperator{\SL}{SL}
\DeclareMathOperator{\PSL}{PSL}
\DeclareMathOperator{\PSU}{PSU}
\DeclareMathOperator{\Mod}{Mod}

\DeclareMathOperator{\sVec}{sVec}
\newcommand{\CC}{\mathcal{C}}
\newcommand{\DD}{\mathcal{D}}
\newcommand{\BB}{\mathcal{B}}
\newcommand{\mfs}{\mathfrak{s}}
\newcommand{\mft}{\mathfrak{t}}
\newcommand{\hC}{\breve{\CC}}
\newcommand{\hT}{\hat{T}}
\newcommand{\hS}{\hat{S}}
\newcommand{\tS}{\tilde{S}}
\newcommand{\ot}{\otimes}
\newcommand{\oA}{\overline{A}}
\newcommand{\one}{\mathbf{1}}
\numberwithin{equation}{section}
\numberwithin{theorem}{section}\numberwithin{example}{section}
\numberwithin{conj}{section}\numberwithin{lem}{section}
\title{Congruence Subgroups and Super-Modular Categories}
\author{Parsa Bonderson}
\email{parsab@microsoft.com}
\address{Microsoft Research Station Q,
    University of California,
    Santa Barbara, CA
    U.S.A.}
\author{Eric C. Rowell}
\email{rowell@math.tamu.edu}
\address{Department of Mathematics,
    Texas A\&M University,
    College Station, TX
    U.S.A.}
\author{Zhenghan Wang}
\email{zhenghwa@microsoft.com}
\address{Microsoft Research Station Q and Department of Mathematics,
    University of California,
    Santa Barbara, CA
    U.S.A.}
    \author{Qing Zhang}
\email{zhangqing@math.tamu.edu}
\address{Department of Mathematics,
    Texas A\&M University,
    College Station, TX
    U.S.A.}
\date{\today}

\begin{document}

\begin{abstract}
A super-modular category  is a unitary pre-modular category with M\"uger center equivalent to the symmetric unitary category of super-vector spaces.  Super-modular categories are important alternatives to modular categories as any unitary pre-modular category is the equivariantization of a either a modular or super-modular category.  Physically, super-modular categories describe universal properties of quasiparticles in fermionic topological phases of matter.  In general one does not have a representation of the modular group $\mathrm{SL}(2,\mathbb{Z})$ associated to a super-modular category, but it is possible to obtain a representation of the (index 3) $\theta$-subgroup: $\Gamma_\theta<\mathrm{SL}(2,\mathbb{Z})$.   We study the image of this representation and conjecture a super-modular analogue of the Ng-Schauenburg Congruence Subgroup Theorem for modular categories, namely that the kernel of the $\Gamma_\theta$ representation is a congruence subgroup.    We prove this conjecture for any super-modular category that is a subcategory of modular category of twice its dimension, i.e. admitting a minimal modular extension.  Conjecturally, every super-modular category admits (precisely 16) minimal modular extensions and, therefore, our conjecture would be a consequence.
\end{abstract}
\maketitle
\thanks{E. Rowell and Q. Zhang were partially supported by NSF grant DMS-1410144, and Z. Wang by NSF grant DMS-1411212.  The authors thank M. Cheng, M. Papanikolas and Z. Sunic for valuable discussions.}
\section{Introduction}
A key part of the data for a modular category $\CC$ is the $S$ and $T$ matrices encoding the non-degeneracy of the braiding and the twist coefficients, respectively.  We will denote by $\tS$ the unnormalized matrix obtained as the invariants of the Hopf link so that $\tS_{0,0}=1$, while $S=\frac{\tS}{D}$ will denote the (unitary) normalized $S$-matrix where $D^2=\dim(\CC)$ is the categorical dimension and $D>0$.  Later, we will use the same conventions for any pre-modular category (for which $S$ may not be invertible).  The diagonal matrix $T:=\theta_i\delta_{i,j}$ has finite order (Vafa's theorem, see \cite{bakalov2001lectures}) for any pre-modular category.  For a modular category the $S$ and $T$ matrices satisfy (see e.g. \cite[Theorem 3.1.7]{bakalov2001lectures}):
\begin{enumerate}
\item $S^2=C$ where $C_{i,j}=\delta_{i,j^*}$ (so $S^4=C^2=I$)
\item $(ST)^3=\frac{D_+}{D}S^2$ where $D_+=\sum_i\tS_{0,i}^2\theta_i$
\item $TC=CT$.
\end{enumerate}
These imply that from any modular category $\CC$ of rank $r$ (i.e. with $r$ isomorphism classes of simple objects) one obtains a projective unitary representation  of the modular group $\rho:SL(2,\Z)\rightarrow \PSU(r)$ defined on generators by:
$\mfs=\begin{pmatrix} 0 & -1\\1&0\end{pmatrix}\rightarrow S$ and $\mft=\begin{pmatrix}1 & 1\\0&1\end{pmatrix}\rightarrow T$ composed with the canonical projection $\pi_r:\U(r)\rightarrow \PSU(r)$.  By rescaling the $S$ and $T$ matrices, $\rho$ may be lifted to a linear representation of $SL(2,\Z)$, but these lifts are not unique.  This representation has topological significance: one identifies the modular group with the mapping class group $\Mod(\Sigma_{1,0})$ of the torus ($\mft$ and $\mfs\mft^{-1}\mfs^{-1}$ correspond to Dehn twists about the meridian and parallel) and this projective representation is the action of the mapping class group on the Hilbert space associated to the torus by the modular functor obtained from $\CC$.

A subgroup $H<\SL(2,\Z)$ is called a \textbf{congruence subgroup} if $H$ contains a principal congruence subgroup $\Gamma(n):=\{A\in \SL(2,\Z): A\equiv I\pmod{n}\}$ for some $n\geq 1$.  Since $\Gamma(n)$ is the kernel of the reduction modulo $n$ map $\SL(2,\Z)\rightarrow \SL(2,\Z/n\Z)$, any congruence subgroup has finite index.  The \textbf{level} of a congruence subgroup $H$ is the minimal $n$ so that $\Gamma(n)<H$. 
More generally, for $G<\SL(2,\Z)$ we say $H<G$ is a congruence subgroup if $G\cap \Gamma(n)<H$ with the level of $H$ defined similarly.  

The connection between topology and number theory found through the representation above is deepened by the following Congruence Subgroup Theorem:

 \begin{theorem}[\cite{ng2010congruence}]\label{nstheorem}
  Let $\CC$ be a modular category of rank $r$ with $T$ matrix of order $N$. Then the projective representation $\rho:\SL(2,\Z)\rightarrow \PSU(r)$ has $\ker(\rho)$ a congruence subgroup of level $N$.
 \end{theorem}
In particular the image of $\rho$ factors over $\SL(2,\Z/N\Z)$ and hence is a finite group. This fact has many important consequences: for example, it is related to rank-finiteness \cite{BNRW1} and can be used in classification problems \cite{BNRW2}.

  A \textbf{super-modular} category is a unitary ribbon fusion category whose M\"uger center is equivalent, as a unitary symmetric ribbon fusion category, to the category $\sVec$ of super-vector spaces (equipped with its unique structure as a unitary spherical symmetric fusion category).   Super-modular categories (or slight variations) have been studied from several perspectives, see \cite{Bon,DMNO,16fold,BCT,KLW} for a few examples.  An algebraic motivation for studying these categories is the following: any unitary braided fusion category is the equivariantization \cite{DGNO} of either a modular or super-modular category (see \cite[Theorem 2]{Sawin}).  Physically, super-modular categories provide a framework for studying fermionic topological phases of matter \cite{16fold}.  Topological motivations include the study of spin 3-manifold invariants (\cite{Sawin,Bl,BM}) and $(3+1)$-TQFTs (\cite{WW}).  
  
  \begin{rmk} We restrict to unitary categories both for mathematical convenience and for their physical significance.  On the other hand, there is a non-unitary version $\sVec^{-}$ of $\sVec$: the underlying (non-Tannakian) symmetric fusion category is the same, but with the other possible spherical structure, which leads to negative dimensions.  We could define super-modular categories more generally as pre-modular categories $\BB$ with M\"uger center equivalent to either of $\sVec$ or $\sVec^{-}$.  However, we do not know of any examples $\BB$ with $\BB^{\prime}\cong\sVec^{-}$ that are not simply of the form $\CC\boxtimes\sVec^{-}$ for some modular category $\CC$ (A. Brugui\`eres asked the second and third authors for such an example in 2016).
  \end{rmk}
 
 One interesting feature of super-modular categories $\BB$ is that their $S$ and $T$ matrices have tensor decompositions (\cite[Appendix]{BNQ},\cite[Theorem III.5]{16fold}):
\begin{equation}\label{stfactored}
S=
\frac{1}{\sqrt{2}}\begin{pmatrix}
	1&1\\1&1\\
	\end{pmatrix}\ot \hS, \quad T= \begin{pmatrix}
	1&0\\0&-1\\
	\end{pmatrix}\ot \hT\end{equation}
    where $\hS$ is unitary and $\hT$ is a diagonal (unitary) matrix, depending on $r/2-1$ sign choices.  Two naive questions motivated by the above are: 1) Do $\hS$ and a choice of $\hT$ provide a (projective) representation of $\SL(2,\Z)$? and 2) Is the group generated by $\hS$ and a choice of $\hT$ finite? Of course if $\BB=\sVec\boxtimes \DD$ for some modular category $\DD$ (\emph{split super-modular}) then the answer to both is yes.  More generally, as Example \ref{exnofinite} below illustrates, the answer to both questions is no.

The physical and topological applications of super-modular categories motivate a more refined question as follows. The consideration of fermions on a torus \cite{MooreVafa} leads to the study of spin structures on the torus $\Sigma_{1,0}$:  there are three even spin structures $(A,A),(A,P),(P,A)$ and one odd spin structure $(P,P)$, where $A,P$ denote antiperiodic and periodic boundary conditions.  The full mapping class group $\Mod(\Sigma_{1,0})=\SL(2,\Z)$ permutes the even spin structures: $\mfs$ interchanges $(P,A)$ and $(A,P)$, and preserves $(A,A)$, whereas $\mft$ interchanges $(A,A)$ and $(P,A)$ and preserves $(A,P)$.  Note that both $\mfs$ and $\mft^2$ preserve $(A,A)$, so that the index $3$ subgroup $\Gamma_\theta:=\langle \mfs,\mft^2\rangle<\SL(2,\Z)$ is the spin mapping class group of the torus equipped with spin structure $(A,A)$.  The spin mapping class group of the torus with spin structure $(A,P)$ or $(P,A)$ is similarly generated by $\mfs^2$ and $\mft$, which is projectively isomorphic to $\Z$.  On the other hand, $\Gamma_\theta$ is projectively the free product of $\Z/2\Z$ with $\Z$ (\cite{Rademacher}).  Now the matrix $\hT^2$ is unambiguously defined for any super-modular category $\BB$, and in \cite[Theorem II.7]{16fold} it is shown that $\mfs\rightarrow \hS$ and $\mft^2\rightarrow\hT^2$ defines a projective representation $\hat{\rho}$ of $\Gamma_\theta$.  We propose the following:

 \begin{conj}\label{mainconj} Let $\BB$ be a super-modular category of rank $2k$ and $\hS$ and $\hT^2$ the corresponding matrices as in equation (\ref{stfactored}).  Then the projective representation  $\hat{\rho}:\Gamma_\theta\rightarrow \PSU(k)$ given by $\hat{\rho}(\mfs)=\pi_k(\hS)$ and $\hat{\rho}(\mft^2)=\pi_k(\hT^2)$ has kernel a congruence subgroup.
   \end{conj}
   
 In particular if this conjecture holds then $\hat{\rho}(\Gamma_\theta)$ is finite.   We do not know what to expect the level of $\ker{\hat{\rho}}$ to be (in terms of, say, the order of $\hT^2$), but we provide some examples below.
 
An important outstanding conjecture (\cite[Question 5.15]{DNO}, \cite[Conjecture III.9]{16fold}, see also \cite[Conjecture 5.2]{M2}) is that every super-modular category $\BB$ has a \emph{minimal modular extension}: that is, $\BB$ can be embedded in a modular category $\CC$ of 
dimension $\dim(\CC)=2\dim(\BB)$.  One may characterize such $\CC$: they are called \emph{spin modular categories} (\cite{BBC}), see Section \ref{spinsection} below.  Our main result proves Conjecture \ref{mainconj} for super-modular categories admitting minimal modular extensions.

\section{Preliminaries}

\subsection{Super-Modular Categories}
Whereas one may always define an $S$-matrix for any ribbon fusion category $\BB$, it may be degenerate.  This failure of modularity is encoded it the subcategory of transparent objects called the \textbf{M\"uger center} $\BB^\prime$. Here an object $X$ is called \textbf{transparent} if all the double braidings with $X$ are trivial: $c_{Y,X}c_{X,Y}=Id_{X\ot Y}$.  By a theorem of Brugui\`eres \cite{Brug} the simple objects in $\BB^\prime$ are those $X$ with $\tS_{X,Y}=d_Xd_Y$ for all simple $Y$, where $d_Y=\dim(Y)=\tS_{\one,Y}$ is the categorical dimension of the object $Y$.  The M\"uger center is obviously \textbf{symmetric}, that is, $c_{Y,X}c_{X,Y}=Id_{X\ot Y}$ for all $X,Y\in\BB^\prime$.  Symmetric fusion categories have been classified by Deligne \cite{Del}, in terms of representations of supergroups.    In the case that $\BB^\prime\cong\Rep(G)$ (i.e. is Tannakian), the modularization (de-equivariantization) procedure of Brugui\`eres \cite{Brug} and M\"uger \cite{M3} yields a modular category $\BB_G$ of dimension $\dim(\BB)/|G|$.  Otherwise, by taking a maximal Tannakian subcategory $\Rep(G)\subset \BB^\prime$ the de-equivariantization $\BB_G$ has M\"uger center $(\BB_G)^\prime\cong \sVec$, the symmetric fusion category of super-vector spaces.  Generally, a braided fusion category $\BB$ with $\BB^\prime\cong\sVec$ as symmetric fusion categories is called \textbf{slightly degenerate} \cite{DGNO}. 

 The symmetric fusion category $\sVec$ has a unique spherical structure compatible with unitarity and has $S-$ and $T-$matrices:
$S_{\sVec}=
\frac{1}{\sqrt{2}}\begin{pmatrix}
	1&1\\1&1\\
	\end{pmatrix}$ and $T_{\sVec}=
\begin{pmatrix}
	1&0\\0&-1\\
	\end{pmatrix}$.

From this point on we will assume that all our categories are unitary, so that $\sVec$ is a unitary symmetric fusion category.  A unitary slightly degenerate ribbon category will be called \textbf{super-modular}.
In other terminology, we say $\BB$ is super-modular if its M\"uger center is generated by a \textbf{fermion}, that is, an object $\psi$ with $\psi^{\otimes 2}\cong\one$ and $\theta_\psi=-1$.

Equation (\ref{stfactored}) shows that the $S$ and $T$ matrices of any super-modular category can be expressed as (Kronecker) tensor products: $S=S_{\sVec}\otimes\hS$ and $T=T_{\sVec}\otimes\hT $ with $\hS$ uniquely determined and $\hT$ determined by some sign choices.
The projective group generated by $\hS$ and $\hT$ may be infinite for all choices of $\hT$ as the following example illustrates:
 \begin{example}\label{exnofinite}
 Consider the modular category $ \SU(2)_{6} $. The label set is $ I = \{0, 1, 2, 3, 4, 5, 6\} $. The subcategory $\PSU(2)_{6}  $ is generated by 4 simple objects with even labels: $ X_{0}= \one, X_{2}, X_{4}, X_{6}$. 　We have $ \hat{S}=\dfrac{1}{\sqrt{4+2\sqrt{2}}}\begin{pmatrix}1 & 1+\sqrt{2}\\ 1+\sqrt{2} & -1\end{pmatrix} $ and $\hat{T}=\begin{pmatrix}1 & 0\\ 0 & \pm i\end{pmatrix} $. For either choice of $\hT$ the eigenvalues of $\hS\hT$ are not roots of unity: one checks that they satisfy the irreducible polynomial $x^{16}-x^{12}+\frac{1}{4}x^8-x^4+1$, which has non-abelian Galois group and is not monic over $\Z$.
 
\end{example}

\subsection{The $\theta$-subgroup of $\SL(2,\Z)$}
The index $3$ subgroup $\Gamma_\theta<\SL(2,\Z)$ generated by $\mfs$ and $\mft^2$ has a uniform description (see e.g. \cite{Koh}):
$$\Gamma_\theta=\{\begin{pmatrix} a& b\\ c&d\end{pmatrix}\in \SL(2,\Z): ac\equiv bd\equiv 0\pmod{2}\}.$$ The notation $\Gamma_\theta$ comes from the fact that Jacobi's $\theta$ series $\theta(z):=\sum_{n=-\infty}^{\infty} e^{n^2\pi \mathrm{i}z}$ is a modular form of weight $1/2$ on $\Gamma_\theta$.
Moreover,$\Gamma_\theta$ is isomorphic to $\Gamma_0(2)$, the Hecke congruence subgroup of level $2$ defined as those matrices in $\SL(2,\Z)$ that are upper triangular modulo $2$, and $\Gamma(2)$ is a subgroup of both $\Gamma_0(2)$ and $\Gamma_\theta$.  In particular $\Gamma_0(2)$ and $\Gamma_\theta$ are distinct, yet isomorphic, congruence subgroups of level $2$.  An explicit isomorphism $\vartheta: \Gamma_\theta\rightarrow \Gamma_0(2)$ is given by $\vartheta(\mathfrak{g})=M\mathfrak{g}M^{-1}$ where $M=\begin{pmatrix} 1 & 1\\0 &2\end{pmatrix}$. 
This can be verified directly, via: $$M\begin{pmatrix} a& b\\ c&d\end{pmatrix}M^{-1}=\begin{pmatrix} a+c&\frac{d+b-a-c}{2}
\\2c&d-c\end {pmatrix}.$$
Observe that $\vartheta(\Gamma(n))=\Gamma(n)$ for any $n$, and for $n$ even $\Gamma(n)\lhd\Gamma_\theta$.  In particular, we see that $\Gamma_\theta/\Gamma(n)<\SL(2,\Z)/\Gamma(n)$ is isomorphic to an index 3 subgroup of $\SL(2,\Z/n\Z)$ that is not normal.  Suppose $\varphi:\Gamma_\theta\rightarrow H$ has kernel a congruence subgroup, i.e. $\Gamma(n)<\ker(\varphi)$.  The congruence level of $\ker(\varphi)$, i.e. the minimal $n$ with $\Gamma(n)<\ker(\varphi)$, is the minimal $n$ so that 
$\Gamma_\theta/\Gamma(n)\twoheadrightarrow \varphi(\Gamma_\theta)$.  The following provides a characterization of such quotients:
\begin{lem}\label{index 3} Suppose that $n=2^kq$ with $k\geq 1$ and $q$ odd. Denote by $P_k$ a $2$-Sylow subgroup of $\SL(2,\Z/2^k\Z)$. Then, $$\Gamma_\theta/\Gamma(n)\cong P_k\times \SL(2,\Z/q\Z).$$
\end{lem}
\begin{proof}
By the Chinese Remainder Theorem, non-normal index $3$ subgroups of $$\SL(2,\Z/n\Z)\cong\prod_{p\mid n}\SL(2,\Z/p^{\ell_p}\Z)$$ correspond to non-normal index $3$ subgroups 
of $\SL(2,\Z/p^{\ell_p}\Z)$ where $n=\prod_{p\mid n}p^{\ell_p}$ is the prime factorization of $n$.  Any $2$-Sylow subgroup of $\SL(2,\Z/2^k\Z)$ has index $3$ and is not normal (since reduction modulo $2$ gives a surjection to $\SL(2,\Z/2\Z)\cong \mathfrak{S}_3$) so it is enough to show that this fails for $\SL(2,\Z/p^{k}\Z)$ with $p>2$. 

In general, if $H<G$ is a non-normal subgroup of index $3$ then the (transitive) left action of $G$ on the coset space $G/H$ provides a homomorphism to the symmetric group on 3 letters: $\phi:G\rightarrow \mathfrak{S}_3$.  If $\phi(G)=\mathfrak{A}_3$ (the alternating group on $3$ letters) then we would have $\ker(\phi)=H\lhd G$.  Thus $\phi(G)=\mathfrak{S}_3$, so that any such group $G$ must have an irreducible $2$ dimensional representation with character values $2,-1,0$.  

By \cite{Nobs,Eholzer} we see that for $p>2$, the groups $\SL(2,\Z/p^k\Z)$ only have $2$-dimensional irreducible representations for $p=3,5$, and each of these representations factor over the reduction modulo $p$ map $\SL(2,\Z/p^k\Z)\twoheadrightarrow\SL(2,\Z/p\Z)$.  By inspection neither $\SL(2,\Z/3\Z)$ nor $\SL(2,\Z/5\Z)$ have $\mathfrak{S}_3$ as quotients.
\end{proof}

\section{Main Results}
In this section we prove Conjecture \ref{mainconj} for any super-modular category that admits a minimal (spin) modular extension.

\subsection{Spin Modular Categories}\label{spinsection}

A \textbf{spin modular category} $\CC$ is a modular category with a (chosen) fermion.  
Let $\CC$ be a spin modular category, with fermion $\psi$, (unnormalized) $S$-matrix $\tS$ and $T$-matrix $T$. Proposition II.3 of \cite{16fold} provides a number of useful symmetries of $\tS$ and $T$:
\begin{enumerate}
\item $\tS_{\psi,\alpha}=\epsilon_\alpha d_\alpha$, where $\epsilon_\alpha=\pm 1$ and $\epsilon_\psi=1$.

\item $\theta_{\psi \alpha}=-\epsilon_\alpha \theta_\alpha$.\label{thetasign}

\item $\tS_{\psi\alpha,\beta}=\epsilon_\beta \tS_{\alpha,\beta}$.\label{smatrixsign} 
\end{enumerate}

We have a canonical $\Z/2\Z$-grading $\CC_0\oplus \CC_1$ with simple objects $X\in\CC_0$ if $\epsilon_X=1$ and $X\in\CC_1$ when $\epsilon_X=-1$.  The trivial component $\CC_0$ is a super-modular category, since $\CC_0^\prime=\langle \psi \rangle\cong \sVec$.  
                                 
Since $\theta_X=-\epsilon_X\theta_{\psi X}$ it is clear that $\psi X\not\cong  X$ for $X\in\CC_0$.  However, objects in $\CC_1$ may be fixed by ${-}\otimes \psi$ or not.  This provides another canonical decomposition $\CC_1=\CC_v\oplus \CC_\sigma$ as abelian categories, where a simple object $X\in\CC_v\subset\CC_1$ if $X\psi\not\cong X$ and 
$X\in\CC_\sigma\subset\CC_1$ if $X\psi\cong X$.  Finally, using the action of ${-}\ot \psi$ we make a (non-canonical) decomposition of  $\CC_0=\hC_0\oplus\psi\hC_0$ and $\CC_v=\hC_v\oplus\psi\hC_v$ so that when $X\in\hC_0$ we have $X\psi\in\psi\hC_0$ and similarly  for $\CC_v$.  Notice that for $X\in\CC_0$ we have $X^*\not\cong\psi\otimes X$ since $\theta_X=\theta_{X^*}$, so that we may ensure $X$ and $X^*$ are both in $\hC_0$ or both in $\psi\hC_0$.  On the other hand, for $Y\in\CC_v$ it is possible that $X^*\cong \psi\otimes X$--for example, this occurs for $SO(2)_1$.

As in \cite{BCT} we choose an ordered basis $\Pi=\Pi_0\bigsqcup \psi\Pi_0\bigsqcup \Pi_v\bigsqcup\psi\Pi_v\bigsqcup\Pi_\sigma$ for the Grothendieck ring of $\CC$ that is compatible with the above partition $\CC=\hC_0\oplus\psi\hC_0\oplus\hC_v\oplus\psi\hC_v\oplus\CC_\sigma$.  Using \cite[Proposition II.3]{16fold} we have the block matrix decomposition for the $S$ and $T$ matrices:

$$S=\begin{pmatrix}\frac{1}{2}\hS & \frac{1}{2}\hS & A&A &X\\
\frac{1}{2}\hS & \frac{1}{2}\hS & -A &-A &-X\\
A^T&-A^T & B& -B&0\\
A^T&-A^T &-B & B &0\\
X^T &-X^T &0 &0 &0
\end{pmatrix}\quad T=\begin{pmatrix} \hat{T}&0& 0&0&0\\
		0&-\hat{T} & 0 &0 &0\\
		0&0 &\hat{T_{v}}& 0&0\\
		0&0 &0 & \hT_{v} &0\\
		0 &0 &0 &0 &{T_{\sigma}}
		\end{pmatrix}.$$
 Here $B$ and $\hS$ are symmetric matrices, and each of $\hT,\hT_v$ and $T_\sigma$ are diagonal matrices.

Now consider the following ordered partitioned basis: 
\begin{enumerate}
\item $\Pi_0^+:=\{X_i+\psi X_i: X_i\in\Pi_0\}$,
\item $\Pi_0^-:=\{X_i-\psi X_i: X_i\in\Pi_0\}$,
\item $\Pi_v^+:=\{Y_i+\psi Y_i: Y_i\in\Pi_v\}$, 
\item $ \Pi_\sigma:=\{Z_i\in\Pi_\sigma\}$ and
\item $\Pi_v^-:=\{Y_i-\psi Y_i: Y_i\in\Pi_v\}$.
\end{enumerate}
  
	With respect to this partitioned basis, the $S$ and $T$ matrices have the block form: 
	$$S^\prime=\begin{pmatrix}\hS & 0& 0&0&0\\
	0 & 0 & 2A &X &0\\
	0&2A^T &0& 0&0\\
	0&2X^T &0 & 0 &0\\
	0 &0 &0 &0 &2B
	\end{pmatrix}\quad T^\prime=\begin{pmatrix}0 &\hat{T}& 0&0&0\\
		\hat{T} & 0 & 0 &0 &0\\
		0&0 &\hat{T_{v}}& 0&0\\
		0&0 &0 & T_{\sigma} &0\\
		0 &0 &0 &0 &\hat{T_{v}}
		\end{pmatrix}.$$
From this choice of basis one sees that the representation $\rho$ restricted to $\Gamma_\theta=\langle \mfs,\mft^2\rangle$ has 3 invariant (projective) subspaces, spanned by $\Pi_0^+,\Pi_0^-\cup \Pi_v^+\cup \Pi_\sigma$ and $\Pi_v^-$ respectively.  In particular we have a surjection $\rho(\Gamma_\theta)\twoheadrightarrow\hat{\rho}(\Gamma_\theta)$, mapping the image of $S$ in $\PSU(|\Pi|)$ to the image of $\hS$ in $\PSU(|\Pi_0^+|)$.
We can now prove:
\begin{theorem} Suppose that $\BB$ is a super-modular category with minimal modular extension $\CC$ so that $\BB=\CC_0$.  Assume further that the $T$-matrix of $\CC$ has order $N$. Then  $\hat{\rho}:\Gamma_\theta\rightarrow \PSU(k)$ has $\ker(\hat{\rho})$ a congruence subgroup of level at most $N$.
\end{theorem}
\begin{proof} Let $S$ and $T$ be the $S$-matrix and $T$-matrix of $\CC$. Consider the projective representation $\rho$ of $\SL(2,\Z)$ defined by $\rho(\mfs)=S$ and $\rho(\mft)=T$.  By Theorem \ref{nstheorem}, $\ker(\rho)$ is a congruence subgroup of level $N$, i.e. $\Gamma(N)<\ker(\rho)$.   Now the restriction of $\rho_{\mid \Gamma_\theta}$ to $\Gamma_\theta$ has $\ker(\rho_{\mid \Gamma_\theta})=\ker(\rho)\cap \Gamma_\theta\supset \Gamma(N)\cap\Gamma_\theta$.  However, since $\CC$ contains a fermion $N$ is even, so $\Gamma(N)<\Gamma(2)<\Gamma_\theta$ hence $\Gamma(N)\cap\Gamma_\theta=\Gamma(N)$.  It follows that $\Gamma(N)<\ker(\rho_{\mid \Gamma_\theta})$.  The discussion above now implies $\Gamma(N)<\ker(\rho_{\mid \Gamma_\theta})<\ker(\hat{\rho})$ as we have a surjection $\rho(\Gamma_\theta)\twoheadrightarrow\hat{\rho}(\Gamma_\theta)$. Thus, we have shown that $\ker(\hat{\rho})$ is a congruence subgroup of level at most $N$, and in particular $\hat{\rho}$ has finite image.

\end{proof}

\subsection{Further Questions}

The charge conjugation matrix $C$ in the basis above has the form $C^\prime_{i,j}=\pm\delta_{i,j^*}$. Since we have arranged that $X_i\in\Pi_0$ implies $X_i^*\in\Pi_0$, $C^\prime_{i,j}=-1$ can only occur for $i=j\in\Pi_v^-$: if $(W-\psi W)^*=-(W-\psi W)$ for some simple object $W$, then $W^*=\psi W$. We see that this can only happen if $W\in\CC_v$ by comparing twists.
Under this change of basis, we have $(S^\prime)^2=\dim(\CC)C^\prime$ and $(S^\prime T^\prime)^3=\frac{D_+}{D}(S^\prime)^2$.  It would be interesting to explore the extra relations among the various submatrices of $S^\prime$ and $T^\prime$.

The 16 spin modular categories of dimension $4$ are of the form $\SO(n)_1$ (where $\SO(n)_1\cong \SO(m)_1$ if and only if $n\cong m\pmod{16}$).  For $n$ odd $\SO(n)_1$ has rank $3$ whereas for $n$ even $\SO(n)_1$ has rank $4$.  For example, the Ising modular category corresponds to $n=1$ and $\SO(2)_1$ has fusion rules like the group $\Z_4$. For any modular category $\DD$ and $1\leq n \leq 16$ the spin modular category $\SO(n)_1\boxtimes \DD$ with fermion $(\psi,\one)$ has either $\CC_\sigma=\emptyset$ or $\CC_v=\emptyset$.  An interesting problem is to classify spin modular categories with either $\CC_\sigma=\emptyset$ or $\CC_v=\emptyset$, particularly those with no $\boxtimes$-factorization.

\section{A Case Study}
Our result gives an upper bound on the level of $\ker(\hat{\rho})$ for super-modular categories $\BB$ with minimal modular extensions $\CC$: the level of $\ker(\hat{\rho})$ is at most the order of the $T$-matrix of $\CC$.  The actual level can be lower: for a trivial example we consider the super-modular category $\sVec$. In this case $\hS=\hT^2=I$ so the level $\ker(\hat{\rho})$ is $1$, yet the order of the $T$ matrix for its (16) minimal modular extensions can be $2,4,8$ or $16$.  More generally for any split super-modular category $\BB=\DD\boxtimes\sVec\subset \DD\boxtimes \SO(n)_1=\CC$ (with fermion $(\one,\psi)$) the ratio of the  levels of the kernels of the $\SL(2,\Z)$ (for $\CC$) and $\Gamma_\theta$ (for $\BB$, i.e. $\DD$) representations can be $2^k$ for $0\leq k\leq 4$.

To gain further insight we consider a family of non-split super-modular categories obtained from the spin modular category (see \cite[Lemma III.7]{16fold})
$ \SU(2)_{4m+2} $. This has modular data:\\
	
	\[\tilde{S}_{i,j}:= \dfrac{\sin \left({\frac{(i+1)(j+1)\pi}{4m+4}}\right)}{\sin ({\frac{\pi}{4m+4}})}, \quad
	T_{j,j}:= e^{\frac{\pi \mathrm{i}(j^{2}+2j)}{8m+8}} \]
	
	where $ 0\leq i,j \leq 4m+2 $.  Since $T$ has order $16(m+1)$, Theorem \ref{nstheorem} implies that the image of the projective representation $\rho:\SL(2,\Z)\rightarrow \PSU(4m+3)$ defined via the normalized $S$-matrix $S$ and $T$ factors over $\SL(2,\Z/N\Z)$ where $N=16(m+1)$. 
    
    The super-modular subcategory $\PSU(2)_{4m+2} $ has simple objects labeled by even $i,j$.  The factorization (\ref{stfactored}) yields the following:
    \begin{equation}\label{PSU2 S and T}\hS_{i,j}= \dfrac{\sin \left({\frac{(2i+1)(2j+1)\pi}{4m+4}}\right)}{\Xi\sin ({\frac{\pi}{4m+4}})}, \quad \hT_{j,j}=e^{\frac{\pi \mathrm{i}(j^2+j)}{2m+2}}\end{equation}
for $0\leq i,j \leq m$, where $\Xi=\frac{\sqrt{\frac{m+1}{2}}}{\sin \left(\frac{\pi }{4 m+4}\right)} $. In \cite{16fold} all $16$ minimal modular extensions of $\PSU(2)_{4m+2}$ are explicitly constructed and each has $T$-matrix of order $16(m+1)$ so that the kernel of the corresponding projective $\SL(2,\Z)$ representation is a congruence subgroup of level $16(m+1)$.  
\iffalse
The image $\hat{\rho}(\Gamma_\theta)$ is somewhat unwieldy in general; computational experiments suggest that $\hat{\rho}(\Gamma_\theta)\cong A_m^\prime\rtimes H$ where $H$ is an abelian group and $A_m^\prime:=[A_m,A_m]$ is the commutator subgroup.  For this reason we study the subgroup $A_m^\prime\lhd \hat{\rho}(\Gamma_\theta)$.  Another well-behaved related group is the quotient $A_m/Z(A_m)$.
Since the representation $\hat{\rho}$ is not irreducible in general, the center $Z(A_m)$ may be larger than the subgroup of scalar matrices in $A_m$.  However we have a surjection $\hat{\rho}(\Gamma_\theta)\twoheadrightarrow A_m/Z(A_m)$.  Moreover, $\hat{\rho}(\Gamma_\theta)$ is a subquotient of $\SL(2,\Z/N\Z)$ where $N=16(m+1)$.  Indeed, if we write $m+1=2^nq$ for some odd number $q$ then the Chinese remainder theorem implies that $\SL(2,\Z/N\Z)\cong \SL(2,\Z/2^{(n+4)}\Z)\times \SL(2,\Z/q\Z)$.  
\fi
Our computations suggests the following conjecture, with cases verified using Magma software \cite{magma} indicated in parentheses.  A sample of the results of these computations are found in Table \ref{table1}.  The notation $\langle n,k\rangle$ indicates the $k$th group of order $n$ in the GAP \cite{GAP} library of small groups.  In the last column, we sometimes give a slightly different description than is indicated in part (f) below.  We include the groups $\hat{\rho}(\Gamma_\theta)$, $A_m^\prime:=[A_m,A_m]$ and $\overline{A}_m:=A_m/Z(A_m)$.  As $\hat{\rho}$ is not necessarily irreducible, we have $\hat{\rho}(\Gamma_\theta)\twoheadrightarrow\overline{A}_m$.  The congruence level of $\ker{\hat{\rho}}$ is computed using Lemma \ref{index 3}.

\begin{conj}\label{bigsuconj}
Let $A_m$ be the subgroup of $\SU(k)$ generated by $\hS$ and $\hT^2$ associated with $\PSU(2)_{4m+2}$, the quotient $\oA_m:=A_m/Z(A_m)$ and the commutator subgroup $A_m^\prime:=[A_m,A_m]$. Then 
\begin{enumerate}
\item[(a)] When $m+1=q$ is odd, $\oA_m=\oA_{q-1}\cong \PSL(2,\Z/q\Z)$ (verified for $2\leq m\leq 18$).
%which has order $q^3\prod_{p|q}\frac{p^2-1}{2p^2}$ ($p$ prime).
\item[(b)] When $m+1=2^n$ we have $|\oA_m|=|\oA_{2^n-1}|=2^{3n+1}$ (verified for $1\leq n\leq 5$).
\item[(c1)] If we write $m+1=2^nq$ where $q$ is odd, then $\oA_m\cong \oA_{2^n-1}\times \oA_{q-1}$ (verified for $1\leq m\leq 14$).
\item[(c2)] If we write $m+1=2^nq$ where $q$ is odd $|\oA_m|=2^{3n+1}q^3\prod_{p|q}\frac{p^2-1}{2p^2}$ (primes $p$) (verified for $1\leq m\leq 21$).
\item[(d)] For $5\leq m+1=p$ prime $A_{p-1}^\prime\cong \SL(2,\Z/p\Z)$ (verified for $4\leq m\leq 12$).
\item[(e)] If we write $m+1=2^nq$ where $q$ is odd, then $A_m^\prime\cong A_{2^n-1}^\prime\times A_{q-1}^\prime$ (verified for $1\leq m\leq 14$).
\item[(f)] For $m+1\not\equiv 0\pmod{4}$, we have $A_m^\prime\lhd\hat{\rho}(\Gamma_\theta)$ and $\hat{\rho}(\Gamma_\theta)$ is an iterated semidirect product of $A_m^\prime$ with cyclic group actions (verified for $1\leq m\leq 14$).  In general, $\ker(\hat{\rho})$ is a congruence subgroup of level $4(m+1)$ (verified for $1\leq m\leq 12$). 
\end{enumerate}
\end{conj}

	\begin{table}\caption{A Sample of $\PSU(2)_{4k+2}$ Results}\label{table1}
		\begin{tabular}{ | c | c | c| c|c|}\hline
			$m$ & $|\oA_m|$ & $\oA_m$&$A_m^\prime$&$\hat{\rho}(\Gamma_\theta)$\\ \hline\hline
			$1$ & $ 2^4 $ &  $ D_{16}$& $\Z/8\Z$&$D_{16}=A_1^\prime\rtimes\Z/2\Z$\\ \hline
			$2$ & $ 12 $ &  $ \PSL(2,\Z/3\Z)$&$\Q_8$&\iffalse$(A_2^\prime\rtimes\Z/3\Z)\rtimes\Z/2\Z$\fi$\SL(2,\Z/3\Z)\rtimes\Z/2\Z$ \\ \hline
			$3$ & $ 2^7$ & $\langle 128,71\rangle$   &$\langle 64,184\rangle$& $\langle 128,71\rangle$\\ \hline
			$4$ & $ 60 $ &  $ \PSL(2,\Z/5\Z) $& $\SL(2,\Z/5\Z)$&$A_4^\prime\rtimes \Z/2\Z$\\ \hline
			$5$ & $ 2^4\cdot 12$ &  $ D_{16}\times \PSL(2,\Z/3\Z)   $ &$\Z/8\Z\times \Q_8$&\iffalse$A_5^\prime\rtimes\Z/6\Z$\fi$(\Z/8\Z\times \SL(2,\Z/3\Z))\rtimes\Z/2\Z$\\ \hline
			$6$ & $ 168$ &  $ \PSL(2,\Z/7\Z) $&$ \SL(2,\Z/7\Z) $ &$A_6^\prime\rtimes \Z/2\Z$\\ \hline
			$7$ & $ 2^{10} $ &  $\overline{A}_{7}   $ &$|\cdot|=2^9$&$\overline{A}_7$\\ \hline
			$8$ & $ 324$ & $\PSL(2,\Z/9\Z) $&$(\Z/3\Z)^3\rtimes \Q_8$&$(A_8^\prime\rtimes\Z/3\Z)\rtimes\Z/2\Z$\\\hline
			$9$ & $ 2^4\cdot 60$ & $ D_{16}\times \PSL(2,\Z/5\Z) $& $ \Z/8\Z\times \SL(2,\Z/5\Z)$&$A_9^\prime\rtimes\Z/2\Z$\\ \hline
			$10$ & $660$ & $\PSL(2,\Z/11\Z) $& $\SL(2,\Z/11\Z)$& $A_{10}^\prime\rtimes\Z/2\Z$ \\\hline
			$11$ & $2^7\cdot 12$ &$\langle 128,71\rangle\times \PSL(2,\Z/3\Z)$  &$\langle64,184\rangle\times\Q_8$&$\SL(2,\Z/3\Z)\rtimes\langle 128,71\rangle$\\ \hline
			$12$ & $1092 $ &  $ \PSL(2,\Z/13\Z) $ &$\SL(2,\Z/13\Z)$&$\SL(2,\Z/13\Z)\rtimes\Z/2\Z$\\ \hline
			$13$ & $2^4\cdot 168 $ &  $D_{16}\times \PSL(2,\Z/7\Z)$&$\Z/8\Z\times \SL(2,\Z/7\Z)$& $A_{13}^\prime\rtimes\Z/2\Z$\\ \hline
			$14$ & $720$ & $ \PSL(2,\Z/15\Z)$&$\Q_8\times \SL(2,\Z/5\Z)$&\iffalse$(A_{14}^\prime\rtimes\Z/3\Z)\rtimes\Z/2\Z$\fi  $\SL(2,\Z/15\Z)\rtimes\Z/2\Z$\\ \hline
		\iffalse	$15$ & $ 2^{13} $  &$   $&& \\ \hline
			$16$ & $ 2448 $  & $\PSL(2,\Z/17\Z)  $&&\\ \hline
			$17$ & $ 2^4\cdot 324 $ & $   $&&\\ \hline
			$18$ & $ 3420$  &$ \PSL(2,\Z/19\Z) $&&\\ \hline
			$19$ & $ 2^7\cdot 60  $& $ $&&\\ \hline
			$20$ & $12\cdot 168  $ &  $  $&&\\ \hline
			$21$ & $ 2^4\cdot 660 $ & $ $ &&\\ \hline\fi
		\end{tabular}
	\end{table}
%The strategy is the following.  1) For \bar{A}_m we compute the group and then look for subgroups of the correct order that are also normal: SG:=Subgroups(G:OrderEqual:=...,Al:="Normal"). To access the groups use SG[i]`subgroup etc. We can look at all pairs H,K and compute Commutator(H,K) which will be trivial if the product is direct. For the simple cases we ask IsSimple(G) 2) For B:=A_m^\prime we use a similar strategy.  3) For the projective image we first look for normal subgroups of the same order as A_m^\prime and then use IsIsomorphic(B,...) to locate a particular subgroup isomorphic to B, hopefully the index is 2.  Alternatively we look for a subgroup X of index 2, and then look for a subgroup Y of order 2 and then use the command X meet Y to find a trivial intersection. Then we know it is a semidirect product.  Now do the same for Y...

\section*{Appendix: Magma Code}
For our computational experiments we used the symbolic algebra software Magma \cite{magma}.  In this appendix we give some basic pseudo-code and some sample Magma code to illustrate how we found the image of $\hat{\rho}(\Gamma_\theta)$ in our case study, so that the interested reader can do similar explorations.
Given an integer $m$, the $(m+1)\times (m+1)$ $\hS$ and $\hT^2$ matrices obtained from $\PSU(2)_{4m+2}$ are given in equation (\ref{PSU2 S and T}).
In order to use the Magma software we express the entries of $\hS$ and $\hT^2$ in the cyclotomic field $\Q(\omega)$, where $\omega$ is an $(8m+8)$-th root of unity.
For this we must write $\sin \left({\frac{(2i+1)(2j+1)\pi}{4m+4}}\right)$ and $\sqrt{2(m+1)}$ in terms of $\omega$ for which we use the result of generalized form of quadratic Gauss sums \cite{berndt1981determination}.

Here is the pseudocode to find $\hat{\rho}(\Gamma_\theta)$ for $\PSU(2)_{4m+2}$:\\
\textbf{algorithm} projective image：\\
\textbf{input:} integer $m$\\
\textbf{output:} $\hat{\rho}(\Gamma_\theta) $\\  

\textbf{set} $ K$ to be the cyclotomic field $\Q(\omega)$, where $\omega$ is an $(8m+8)$-th root of unity.\\
\textbf{set} $M=2(m+1)$\\
\textbf{initialize} S and T2 to be $(m+1)\times (m+1)$ zero matrices over K.\\
\textbf{initialize} $\alpha=0$.\\

\textbf{step 1: calculate $\alpha$} \\
\textbf{if} $M\equiv 0\pmod{4}$  \textbf{return} $\alpha=\sum_{n=0}^{M-1} \omega^{4n^2}/(1+\omega^M)$\\
\textbf{else} Consider $M/2=m+1 \pmod{4}$. Notice there are only two cases: $m+1\equiv 1$ (mod 4) and $m+1\equiv 3\pmod{4}$.\\
\textbf{if} $m+1\equiv 1 \pmod{4}$  \textbf{return} $\alpha=\frac{\omega^{m + 1} - \omega^{-(m + 1)}}{\omega^{2m + 2}}\sum_{n=0}^{m} \omega^{8n^2}$\\
\textbf{else} \textbf{return} $\alpha=\frac{\omega^{m + 1} - \omega^{-(m + 1)}}{\omega^{2m + 2}}\sum_{n=0}^{m} \omega^{8n^2}/(\omega^M)$ \\
%\textbf{else} $M\equiv 2\pmod{4}$ consider $M/2\pmod{4}$. Notice there are only two cases: $M/2\equiv 1\pmod{4}$ and .\\
\textbf{return} $\alpha=2/\alpha$.
\\

\textbf{step 2: define the entries} \\
\textbf{for} $1\leq i,j\leq m+1$, $S_{i,j}=\alpha \dfrac{\omega^{(2i-1)(2j-1)}-\omega^{-(2i-1)(2j-1)}}{2(\omega^M)}$ and 
$T2_{j,j}=\omega^{(2(j-1))^2+4(j-1)}$\\

\textbf{step 3: find the projective image} \\
\textbf{set} $A$ to be the matrix group generated by $S$ and $T2$ defined above, and $ZK$ the group of scalar matrices over $K$. The projective image of $A$ is then $A/(ZK\cap A)$. \\

The following code can be used in Magma \cite{magma} to find the $\hat{\rho}(\Gamma_\theta)$ in this case, and slight modifications will give the other headings of Table \ref{table1}:
\begin{verbatim}
m:=1;
K<w>:= CyclotomicField(8*m+8);
GL:=GeneralLinearGroup(m+1,K);
M:=2*(m+1);
alpha:=0;
if M mod 4 eq 0 then
for n:=0 to M-1 do
alpha:=alpha + w^(4*(n^2));
end for;
alpha:=alpha/(w^M+1);
else
if (m+1) mod 4 eq 1 then
for n:=0 to m do
alpha:= alpha + w^(8*n^2);
end for;
else
for n:=0 to m do
alpha:=alpha + w^(8*(n^2));
end for;
alpha:=alpha/(w^M);
end if;
alpha:=((w^(m + 1) - w^(-(m + 1)))/(w^(2*m + 2)))*alpha;
end if;
alpha:=2/alpha;
S:=ZeroMatrix(K,m+1,m+1);
for i:=1 to m+1 do
for j:=1 to m+1 do
S[i,j]:=(w^((2*i-1)*(2*j-1))-w^(-(2*i-1)*(2*j-1)))/(2*(w^M));
S[i,j]:=S[i,j]*alpha;
end for;
end for;
T2:=ZeroMatrix(K,m+1,m+1);
for j:=1 to m+1 do
T2[j,j]:=w^((2*(j-1))^2+4*(j-1));
end for;
A:=MatrixGroup<m+1,K|S,T2>;
ZK:=MatrixGroup<m+1,K|w*IdentityMatrix(K,m+1)>;
F:=(A/(A meet ZK));
\end{verbatim}
%B:=MatrixGroup<m+1,K|S^(-1)*T2^(-1)*S*T2, S^(-1)*T2^(-2)*S*T2^2,
%S^(-1)*T2^(-3)*S*T2^3, S^(-1)*T2^(-4)*S*T2^4, S^(-1)*T2^(-5)*S*T2^5>;Order(B);
\bibliographystyle{abbrv}
\bibliography{bibliography.bib}
\end{document}